\documentclass[12pt]{amsart}

\usepackage{amsmath,amssymb,epic,lscape,enumerate,textcomp}

\newtheorem{theorem}{Theorem}[section]
\newtheorem{proposition}[theorem]{Proposition}
\newtheorem{lemma}[theorem]{Lemma}
\newtheorem{corollary}[theorem]{Corollary}

\theoremstyle{definition}
\newtheorem{definition}[theorem]{Definition}

\theoremstyle{remark}

\numberwithin{equation}{section}
\providecommand{\bysame}{\leavevmode\hbox to3em{\hrulefill}\thinspace}

\def\DJo{$\;$\kern-.4em
    \hbox{D\kern-.8em\raise.15ex\hbox{--}\kern.35em okovi\'{c}}}

\def\al{{\alpha}}

\def\ve{{\varepsilon}}
\def\vf{{\varphi}}

\def\sig{{\sigma}}

\def\bZ{{\mbox{$\mathbb Z$}}}

\def\pE{{\mathcal E}}

\def\Zp{{\makebox[.75em]{+}}}
\def\Zm{{\makebox[.75em]{-}}}

\renewcommand{\subjclassname}{\textup{2000} Mathematics Subject
Classification }

\begin{document}

\title[Turyn-type sequences]
{Turyn-type sequences: Classification, Enumeration and Construction}

\author[D. Best]
{D.  Best}

\author[D.\v{Z}. \DJo]
{D.\v{Z}. \DJo}


\author[H. Kharaghani]
{H. Kharaghani}

\author[H. Ramp]
{H. Ramp}

\address{Department of Pure Mathematics and Institute for Quantum Computing, University of Waterloo,
Waterloo, Ontario, N2L 3G1, Canada}

\email{djokovic@uwaterloo.ca}

\address{Department of Mathematics and Computer Science, University of Lethbridge,
Lethbridge, Alberta, T1K 3M4, Canada}

\email{darcy.best@uleth.ca,hadi@cs.uleth.ca,hugh.ramp@uleth.ca}



\date{}

\begin{abstract}
Turyn-type sequences, $TT(n)$, are quadruples of $\{\pm1\}$-sequences $(A;B;C;D)$, with lengths $n,n,n,n-1$ respectively, where the sum of the nonperiodic autocorrelation functions of $A,B$ and twice that of $C,D$ is a $\delta$-function (i.e., vanishes everywhere except at 0).   
Turyn-type sequences $TT(n)$ are known to exist for all even $n$ not larger than 36.
We introduce a definition of equivalence to construct a canonical
form for $TT(n)$ in general. By using this canonical form,
we enumerate the equivalence classes of $TT(n)$ for $n \le 32$.
We also construct the first example of Turyn-type sequences $TT(38)$. \end{abstract}

\maketitle
\subjclassname{ 05B20, 05B30 }
\vskip5mm

\section{Introduction} \label{Uvod}
Let a {\em binary} {\em sequence} be a sequence $A=a_1,...,a_m$ whose terms belong to $\{\pm 1\}$.
To such a sequence, we associate the polynomial
$A(x)=a_1+a_2x+\cdots+a_mx^{m-1}$, and refer to the Laurent polynomial $N(A)=A(x)A(x^{-1})$ as
the {\em norm} of $A$.
Denoted $TT(n)$, a {\em Turyn-type sequence} $(A;B;C;D)$ is a quadruple of binary sequences with $A,B$ and $C$ of length $n$ and $D$ of length $n-1$, such that
\begin{equation} \label{norm-TT}
N(A)+N(B)+2N(C)+2N(D)=6n-2.
\end{equation}

Turyn-type sequences should not be confused with the so called 
``Turyn sequences'' \cite[Definition 5.1, p. 478]{SY}, which are 
also quadruples of $\{\pm1\}$-sequences of lengths 
$n,n,n-1,n-1$. In addition to the requirement that the sum of 
their non-periodic autocorrelation functions is a $\delta$-function, 
they also have certain desirable symmetry properties. 
Unfortunately, there are only a few known Turyn sequences, 
all with $n\le14$.

Turyn-type sequences play an important role in the construction of Hadamard matrices \cite{HCD,SY}. For instance, the
discovery of a Hadamard matrix of order 428 \cite{KT} used a $TT(36)$,
constructed specifically for that purpose. 
From $TT(n)$, one can construct (as explained in Section 5) base sequences of lengths 
$2n-1$, $2n-1$, $n$, $n$. 
If base sequences of lengths $m$, $m$, $n$, $n$ are known, 
one can use the Goethals-Seidel array to construct a Hadamard 
matrix of order $4(m+n)$. We refer the reader to \cite[p. 436]{KT} for 
details.

Furthermore, two of the three remaining orders less than 1000 for which the existence of a Hadamard matrix is not known may be resolved by using Turyn-type sequences of appropriate lengths (assuming that 
they exist).  
More precisely, Turyn-type sequences $TT(56)$ and $TT(60)$ may be used to construct Hadamard matrices of orders 668 and 716 respectively.

The discovery of any new Turyn-type sequences leads to an infinite class of Hadamard matrices, as explained in  \cite[p. 439]{KT}. Despite the importance of Turyn-type sequences, not much is known about their existence. All the existing results related to these sequences rely on increasingly lengthy computer calculations. In order to have a better understanding of  the structure of Turyn-type sequences, it is essential to classify them for as many values of  $n$ as possible.
Our main goal is to provide a classification of $TT(n)$ for even $n\le32$ ($TT(n)$ do not exist for odd $n>1$) and to modify an existing search method to construct a $TT(38)$. The new $TT(38)$ can be used to construct an infinite class of Hadamard matrices; see \cite[p. 439]{KT}.

In Section \ref{TT-Niz}, we define the standard elementary
transformations of $TT(n)$ and use them to introduce an
equivalence relation. We also introduce a canonical form for
Turyn-type sequences. Using this, we are able to compute the
representatives of the equivalence classes.

An abstract group of order $2^{10}$ is introduced in Section
\ref{Grupa}, which acts naturally on all sets of $TT(n)$.
The orbits of this group are the equivalence classes of $\{TT(n)\}$.

In Section \ref{Tablice}, a list of representatives of the equivalence classes of $\{TT(n)\}$ (those for even $n\le 32$) are tabulated. 
Due to their excessive length, the tables for $n>10$ are truncated
to 12 members only.

Finally, in Section \ref{tt(38)}, the search method for finding a  $TT(38)$ is explained.

\section{A canonical form for turyn-type sequences} \label{TT-Niz}

We denote finite sequences of integers by capital letters.
If $A$ is such a sequence of length $n$, then we denote
its elements by the corresponding lower case letters. Thus,
$$ A=a_1,a_2,\ldots,a_n. $$
The {\em nonperiodic autocorrelation function} of $A$, $N_A$, is
defined by:
$$ N_A(i)=\sum_{j\in\bZ} a_ja_{i+j},\quad i\in\bZ, $$
where $a_k=0$ for $k<1$ and $k>n$. 
 (As usual, $\bZ$ denotes the ring of integers.)
Note that $N_A(-i)=N_A(i)$ for all $i\in\bZ$ and $N_A(i)=0$ 
for $i\ge n$. 
The integers $N_A(i)$ are the coefficients of the norm of $A$, 
i.e., we have
$$ N(A)=\sum_{i\in\bZ} N_A(i) x^i . $$

Assume that $(A;B;C;D)$ is a $TT(n)$. 
From equation (\ref{norm-TT}), we have
\begin{equation} \label{KorNula}
N_A(i)+N_B(i)+2N_C(i)+2N_D(i)=0, \quad i\ne0.
\end{equation}
The {\em negated} sequence, $-A$, the {\em reversed} sequence, $A'$, and the {\em alternated} sequence,
$A^*$, of the sequence $A$ are defined by
\begin{eqnarray*}
-A &=& -a_1,-a_2,\ldots,-a_n, \\
A' &=& a_n,a_{n-1},\ldots,a_1, \\
A^* &=& a_1,-a_2,a_3,-a_4,\ldots,(-1)^{n-1}a_n
\end{eqnarray*}
respectively. Observe that $N(-A)=N(A')=N(A)$ and $N_{A^*}(i)=(-1)^i N_A(i)$
for all $i\in\bZ$.

We define four types of elementary transformations
of Turyn-type sequences. 	

The {\em elementary transformations} of $(A;B;C;D)\in \{TT(n)\}$
are the following:

(T1) Negate one of $A,B,C $ or $ D$.

(T2) Reverse one of $A, B, C$ or $D$.

(T3) Alternate all four sequences $A, B, C$ and $D$.

(T4) Interchange the sequences $A$ and $B$.


We say that two $TT(n)$ are {\em equivalent} if
one can be transformed to the other by applying a finite
sequence of elementary transformations.
One can enumerate the equivalence classes by finding suitable
representatives of the classes.
For that purpose, we introduce a canonical form.

\begin{definition} \label{KanFor}
We say that $S=(A;B;C;D) \in \{TT(n)\}$ is in 
{\em canonical form} if the following six conditions hold:
\begin{enumerate}[(i)]
 \item $a_1=a_n=b_1=b_n=c_1=d_1=+1$;
 \item If $i$ is the least index such that $a_i\ne a_{n+1-i}$, then $a_i=+1$;
\item If $i$ is the least index such that $b_i\ne b_{n+1-i}$, then $b_i=+1$;	
\item	If $i$ is the least index such that $c_i=c_{n+1-i}$,
then $c_i=+1$;
\item	  If $i$ is the least index such that
$d_i d_{n-i}\ne d_{n-1}$, then $d_i=+1$;
\item Assume that $n>2$. If $a_2\ne b_2$ then $a_2=+1$; if $a_2 = b_2$ then 
$a_{n-1}=+1$ and $b_{n-1}=-1$.
\end{enumerate}
\end{definition}

Note that if $n>1$, then (i) and (\ref{KorNula}) imply 
that $c_n=-1$.
We can now prove that each equivalence class has a member
which is in the canonical form. The uniqueness of this member
will be proved in the next section.

\begin{proposition} \label{Klase}
Each equivalence class $\pE\subseteq \{TT(n)\}$ has at least
one member having the canonical form.
\end{proposition}
\begin{proof}
Let $S=(A;B;C;D)\in\pE$ be arbitrary. By applying the first three
types of elementary transformations, we can assume that (i) holds.
To satisfy the condition (ii), replace $A$ with $A'$ (if necessary),
and similarly, we can satisfy the condition (iii).
To satisfy the condition (iv), replace $C$ with $-C'$ (if necessary).

To satisfy (v), observe that if $i$ exists, then $D$ is not symmetric and  $1<i \le n/2$. If $d_{n-1}=+1$, it suffices to replace
$D$ with $D'$ (if necessary). Otherwise, we replace $D$ with $-D'$
(if necessary).

To satisfy (vi), observe that the condition (\ref{KorNula}) with 
$i=n-2$ implies that exactly one of the equalities $a_2=b_2$ and 
$a_{n-1}=b_{n-1}$ hold. Thus, it suffices to apply T4 (if necessary). 
Hence, $S$ is now in the canonical form.
\end{proof}

\section{A symmetry group of $\{TT(n)\}$} \label{Grupa}

We shall construct a group $G$ of order $2^{10}$ which acts 
naturally on all $\{TT(n)\}$. Our (redundant) generating set for
$G$ will consist of 10 involutions. Each of these generators
is an elementary transformation, and we use
this information to construct $G$, i.e., to impose the defining
relations. Let $S=(A;B;C;D)$ be an arbitrary member of
$\{TT(n)\}$.

To construct $G$, we start with an elementary abelian group $E$
of order $2^8$ with generators $\nu_i,\rho_i$, $i\in\{1,2,3,4\}$.
It acts on $\{TT(n)\}$ as follows:
\begin{eqnarray*}
&& \nu_1S=(-A;B;C;D),\quad \rho_1S=(A';B;C;D), \\
&& \nu_2S=(A;-B;C;D),\quad \rho_2S=(A;B';C;D), \\
&& \nu_3S=(A;B;-C;D),\quad \rho_3S=(A;B;C';D), \\
&& \nu_4S=(A;B;C;-D),\quad \rho_4S=(A;B;C;D').
\end{eqnarray*}
That is, $\nu_i$ negates the $i$th sequence of $S$ and $\rho_i$
reverses it.

Next we introduce the involutory generator $\sig$.
We declare that $\sig$ commutes with
$\nu_3,\nu_4,\rho_3,\rho_4$, and that
$$
\sig\nu_1=\nu_2\sig,\quad
\sig\rho_1=\rho_2\sig.
$$
The group $H=\langle E,\sig \rangle$ is
the direct product of the group
$H_1= \langle \nu_1,\rho_1,\sig \rangle$ of order 32
and $H_2=\langle \nu_3,\nu_4,\rho_3,\rho_4 \rangle$.
The action of $E$ on $\{TT(n)\}$ extends to $H$ by defining
$\sig S=(B;A;C;D)$.

Finally, we define $G$ as the semidirect product of
$H$ and the group of order 2 with generator $\al$.
By definition, $\al$ satisfies $\al\rho_i\al=\rho_i\nu_i$ and commutes with $\rho_4$, $\sig$ and each $\nu_i$, for $i=1,2,3$.
The action of $H$ on $\{TT(n)\}$ extends to $G$ by letting $\al$ act
as the elementary transformation (T3), i.e., we have
$$ \al S=(A^*;B^*;C^*;D^*). $$

We point out that the definition of $G$ is independent of $n$.

The following proposition follows immediately from the
construction of $G$ and the description of its action on $\{TT(n)\}$.

\begin{proposition} \label{Orbite}
The orbits of $G$ in  $\{TT(n)\}$ are the same as the
equivalence classes.
\end{proposition}

We shall need the following lemma.
\begin{lemma} \label{Lema}
For $S=(A;B;C;D)\in  \{TT(n)\}$, set $\vf(S)=a_1 a_n$. Then we
have $\vf(\al S)=-\vf(S)$ and $\vf(hS)=\vf(S)$ for all $h\in H$.
\end{lemma}
\begin{proof}
The first assertion holds because $n$ is even. To prove the
second assertion, it suffices to verify that it holds when
$h$ is one of the generators $\nu_j,\rho_j$, $j=1,2,3,4$, or
$\sig$. This is obvious in the former case. It is also true
in the latter case $(h=\sig)$ because equation (\ref{KorNula}) 
with $i=n-1$ implies that $a_1 a_n=b_1 b_n$.
\end{proof}

The main tool that we use to enumerate the equivalence classes of
 $\{TT(n)\}$ is the following theorem.

\begin{theorem} \label{Glavna}
For each equivalence class $\pE\subseteq  \{TT(n)\}$	 there is a
unique $S=(A;B;C;D)\in\pE$ having the canonical form.
\end{theorem}
\begin{proof}
In view of Proposition \ref{Klase}, we just have to prove the
uniqueness assertion. Let
$$
S^{(k)}=(A^{(k)};B^{(k)};C^{(k)};D^{(k)})\in\pE,\quad (k=1,2)
$$
be in the canonical form. We have to prove that in fact
$S^{(1)}=S^{(2)}$.

By Proposition \ref{Orbite}, we have $gS^{(1)}=S^{(2)}$ for some
$g\in G$. We can write $g$ as $g=\al^t h$ where $t\in\{0,1\}$ and
$h\in H$. The symbols (i)-(vi) will refer to the corresponding
conditions of Definition \ref{KanFor}.

Since both $S^{(1)}$ and $S^{(2)}$ have the canonical form,
the condition (i) implies that $\vf(S^{(1)})=\vf(S^{(2)})=1$,
where $\vf$ is the function defined in Lemma \ref{Lema}.
Now this lemma implies that $t=0$, i.e., $g=h\in H$.

Recall that $H=H_1\times H_2$. Thus, $g=h=h_1h_2$ with
$h_1\in H_1$ and $h_2\in H_2$.
Consequently, $h_2 C^{(1)}=C^{(2)}$ and $h_2 D^{(1)}=D^{(2)}$. 
[What we really mean by these equations is that we have 
$$ h_2S^{(1)}=(A^{(1)};B^{(1)};C^{(2)};D^{(2)}). $$ 
Hopefully this simplified notation for the action of $H_2$, 
as well as its analog for the action of $H_1$, 
will not lead to any confusion.]
We can write $h_2=\nu_3^p \rho_3^q \nu_4^r \rho_4^s$ for some
$p,q,r,s\in\{0,1\}$. Then we have
 $\nu_3^p \rho_3^q  C^{(1)}=C^{(2)}$ and
$\nu_4^r \rho_4^s D^{(1)}=D^{(2)}$.
We shall now prove that $C^{(1)}=C^{(2)}$ and $D^{(1)}=D^{(2)}$.

Since $c^{(1)}_1=c^{(2)}_1=1$ and $c^{(1)}_n=c^{(2)}_n=-1$,
we conclude that $p=q$. Now the condition (iv) implies that
either $p=q=0$ or $\nu_3\rho_3 C^{(1)}=C^{(1)}$. In both
cases we have $C^{(1)}=C^{(2)}$.

Since $d^{(1)}_1=d^{(2)}_1=1$, the equality
$\nu_4^r \rho_4^s D^{(1)}=D^{(2)}$ implies that
$d^{(1)}_1d^{(1)}_{n-1}=d^{(2)}_1d^{(2)}_{n-1}$. Hence 
$d^{(1)}_{n-1}=d^{(2)}_{n-1}=\ve$.  If $\ve=+1$ we must have
$r=0$ and the condition (v) shows that either $s=0$ or
$\rho_4 D^{(1)}=D^{(1)}$. In both cases, $D^{(1)}=D^{(2)}$.
By a similar argument as in the previous paragraph but using 
the condition (v) instead of (iv), we
can show that this equality also holds when $\ve=-1$.

It remains to prove that $A^{(1)}=A^{(2)}$ and $B^{(1)}=B^{(2)}$.
Since $a_1^{(1)}=a_{n}^{(1)}=a_1^{(2)}=a_{n}^{(2)}=+1$, we must
have $h_1\in\langle \rho_1,\rho_2,\sig \rangle$, i.e.,
$h_1=\rho_1^u \rho_2^v \sig^w$ for some $u,v,w\in\{0,1\}$.
We claim that we can assume, without any loss of generality,
that $w=0$. This is clear if $A^{(1)}=B^{(1)}$. Otherwise, we 
have $n>2$ and the condition (vi) implies that either 
$a^{(1)}_{n-1}=+1$, $b^{(1)}_{n-1}=-1$ 
and $a^{(1)}_{n-1}=b^{(1)}_{n-1}$ or $a^{(1)}_{n-1}=+1$, 
$b^{(1)}_{n-1}=-1$ and $a^{(1)}_1=b^{(1)}_1$. 
It is now easy to see that we must have $w=0$. 
This proves our claim, and so we may assume that 
$h_1=\rho_1^u \rho_2^v$.
Consequently, we have $\rho_1^u A^{(1)}=A^{(2)}$ and
$\rho_2^v B^{(1)}=B^{(2)}$. The condition (ii) implies
that either $u=0$ or $\rho_1 A^{(1)}=A^{(1)}$. In both cases
we have $A^{(1)}=A^{(2)}$. The proof of $B^{(1)}=B^{(2)}$
is similar, using (iii) instead of (ii).
\end{proof}

\section{Representatives of the equivalence classes} \label{Tablice}

We have computed a set of representatives for the equivalence
classes of Turyn-type sequences for even $n\le32$. Due to
their excessive size, we tabulate whole sets for only $n\le12$.
Each representative is given in the canonical form,
which is made compact by using the following standard encoding scheme for Turyn-type sequences.

Let $S=(A;B;C;D)\in  \{TT(n)\}$.
For each index $i=1,2,\ldots,n-1$ the number
$4(1-a_i)+2(1-b_i)+(1-c_i)+(1-d_i)/2$
is an integer in the range $0,1,\ldots,15$.
We shall replace this integer by the
corresponding hexadecimal digit $h_i\in\{0,1,\ldots,9,a,b,c,d,e,f\}$.
We encode $S$ by the sequence $h_1,h_2,\ldots,h_n$ of $n$
hexadecimal digits. The hexadecimal digit $h_n$ represents 
the number $2(1-a_n)+(1-b_n)+(1-c_n)/2\in\{0,1,\ldots,7\}$.

Equivalently,  if we apply the substitution $+1\to 0$, $-1\to 1$
to the sequence $a_i,b_i,c_i,d_i$ for $i<n$, and the
sequence $a_n,b_n,c_n$ for $i=n$,
then we obtain the binary representation of the hexadecimal 
digit $h_i$. 
Clearly, the encoding sequence $h_1,h_2,\ldots,h_n$ of $S$ 
determines $S$ uniquely.

As an example, the Turyn-type sequence
{\small
\begin{eqnarray*}
A &=& \Zp\Zp\Zm\Zp\Zm\Zp\Zm\Zp; \\
B &=& \Zp\Zm\Zm\Zm\Zm\Zm\Zm\Zp; \\
C &=& \Zp\Zm\Zm\Zp\Zp\Zp\Zp\Zm; \\
D &=& \Zp\Zp\Zp\Zm\Zp\Zp\Zm
\end{eqnarray*}
}
is encoded as $06e5c4d1$. Note that when displaying a binary sequence, we shall often
write $+$ for $+1$ and $-$ for $-1$.

For each $n$, the representatives are listed in the lexicographic
order of the symbol sequences $h_1,h_2,\ldots,h_n$. Since all
representatives have the canonical form, we always have $h_1=0$
and $h_n=1$. In tables 2-3, the last hexadecimal digit $h_n=1$ is omitted. However, the first hexadecimal digit $h_1=0$ will always be recorded.

For $n\le10$ we list all representatives in Table 2. For
$12 \le n \le 32$, we list in each case only the first dozen 
representatives. For $n \le 22$, the list of representatives
was computed independently by two different programs,
but for the range $24 \le n \le 32$, only the more optimized program
was used. We discuss the search method in the next section.

\begin{center}
\begin{tabular}{c|c|c|c|c|c|c|c|c}
\multicolumn {9}{c}{Table 1: The number of equivalence classes in $\{TT(n)\}$} \\ \hline \hline

 n            & 2 & 4 & 6 & 8 & 10 & 12  & 14  & 16 \\
$|\{TT(n)\}|$ & 1 & 1 & 4 & 6 & 43 & 127 & 186 & 739  \\
\hline
n             & 18  & 20  & 22   & 24   & 26   & 28   & 30   & 32 \\
$|\{TT(n)\}|$ & 675 & 913 & 3105 & 3523 & 3753 & 4161 & 4500 & 6226 \\
\end{tabular}
\end{center}

\ttfamily

\begin{center}

\begin{tabular}{rlrlrlrl}
\multicolumn{8}{c}{Table 2: Class representatives for
$n=2,4,6,8,10$} \\ \hline

\multicolumn{8}{c}{$n=2$} \\
1 & 0  & & \\ \hline
\multicolumn{8}{c}{$n=4$} \\
1 & 016  & & \\ \hline
\multicolumn{8}{c}{$n=6$} \\
1 & 006d6  & 2 & 01396  & 3 & 045ec  & 4 & 0608d  \\ \hline
\multicolumn{8}{c}{$n=8$} \\
1 & 001c6a5  & 2 & 0049e25  & 3 & 005e5c6  & 4 & 00c1786  \\
5 & 06e054d  & 6 & 06e5c4d  & \\  \hline
\multicolumn{8}{c}{$n=10$} \\

1 & 0001f4a96  & 2 & 00036c796  & 3 & 0006f8365  & 4 & 000ef86a5  \\
5 & 00134e696  & 6 & 001ce8965  & 7 & 0047e4f16  & 8 & 0049a13c6  \\
9 & 0057c6e16  & 10 & 0076f4ee5  & 11 & 007809cd6  &
12 & 007b393e5 \\
13 & 007cc94d6  & 14 & 007cca8e5  & 15 & 00870bec6
  & 16 & 008f4dac6  \\
17 & 00b6fa2e5  & 18 & 00c5c7e85  & 19 & 00e063895
  & 20 & 00f6e8ea5  \\
21 & 012408f96  & 22 & 01402b8e5  & 23 & 014308ae5
  & 24 & 0401368bc  \\
25 & 044a18fec  & 26 & 04932a63c  & 27 & 05176df5c
  & 28 & 052bb137c  \\
29 & 05716d9dc  & 30 & 0588caf1c  & 31 & 05a82aedc
  & 32 & 05b7b13dc  \\
33 & 05bf1b5dc  & 34 & 05fb71f5c  & 35 & 061137b4d
  & 36 & 06113b58d  \\
37 & 0614aec8d  & 38 & 061ae6e8d  & 39 & 061b3738d
  & 40 & 061d7f54d  \\
41 & 06a1058cd  & 42 & 06bcd84cd  & 43 & 074625ccd  &  \\ \hline

\end{tabular}
\\
\end{center}

\begin{center}

\begin{tabular}{rlrl}
\multicolumn{4}{c}{Table 3: First twelve class representatives}\\
\multicolumn{4}{c}{for $n=12,14,16,18,20,22,24,26,28,30,32$}
\\ \hline

\multicolumn{4}{c}{$n=12$} \\
1 & 0004f90bc96  & 2 & 0006b8c1da5 \\
3 & 0007c918e96  & 4 & 0008bd43c96 \\
5 & 0009e0a7c95  & 6 & 000b0f68d66 \\
7 & 000b8d50e96  & 8 & 000d26db4a6 \\
9 & 000d2e974a6  & 10 & 000d2e978a6 \\
11 & 000e471ea96 & 12 & 000f0736695 \\

\multicolumn{4}{c}{$n=14$} \\
1 & 00036ac71c765  & 2 & 00041f906bca5 \\
3 & 000497813eca5  & 4 & 0006698fc23a5  \\
5 &0007b2af4e3a5   & 6 & 0007b2b343e95 \\
7 &0008e783d62a5 & 8 & 000a07d41ad96\\
9 &  000af2175a396 & 10 & 000b31c7563a5\\
 11 &000b6283acd65& 12 & 000b679e32ea5\\

\multicolumn{4}{c}{$n=16$} \\
1 &0000778e52de556 & 2 &00007e4b0e53956 \\
3 &0000f0d734a5966 & 4 & 0000f5461f2a965\\
5 &0000fdc397459a5 & 6 &0000fdc397499a5 \\
7 & 00018f07d45ea95& 8 &0001c39c6e95965 \\
9 &00023e1c6748795 & 10 &00023e1c6b48755 \\
11 &00049b15e4d3ca5 & 12 & 0004fe172b471a6\\

\multicolumn{4}{c}{$n=18$} \\
1  & 00006758b30d1e9a5 & 2  & 0000b7c117952e9a5 \\
3  & 0000f87341bd29956 & 4  & 000149f0b259ee595 \\
5  & 00017c2183a68f655 & 6  & 0001897a4c3df0596  \\
7  & 0001b465432e0fa95 & 8  & 0001cb44731d2a9a5 \\
9  & 00030e9bb21da8b65 & 10 & 000363645f0e52b95  \\
11 & 000366969e231c755 & 12 & 0004b350d6918f1a6 \\

\multicolumn{4}{c}{$n=20$} \\
1  & 000038e2739c7a0b695 & 2  & 00004ef0b7c0b6bc5a6  \\
3  & 00006bcab161e913a65 & 4  & 000077078d6f2433a95  \\
5  & 0000bb0754e3e523695 & 6  & 0000bf40b3a3938d696  \\
7  & 0000cb30fe68a5f86a5 & 8  & 0000e8af34cb43e95a6  \\
9  & 0000f04b72a1f196a65 & 10 & 0000f0b27acca39a695  \\
11 & 0000f0f216c9ba59aa6 & 12 & 0000f0f4ce7a15aa966\\

\multicolumn{4}{c}{$n=22$} \\
1  & 00000f702c71a9ad56596 & 2  & 00000f7a12bd68e36a596  \\
3  & 00000f8b358d263c5aa56 & 4  & 00000fb60539ea1ea69a5  \\
5  & 000032f0f792e9665b966 & 6  & 000032f87f835e2a57966  \\
7  & 00003609cf34a6d81b9a5 & 8  & 0000376a43258e2dcb965  \\
9  & 00003fa2c8bc24bd47a56 & 10 & 00004f1acf9149e8b2596  \\
11 & 0000538b4dea91227c696 & 12 & 000058782f506bd31c966\\
\end{tabular}

\begin{tabular}{rcrc}
\multicolumn{4}{c}{$n=24$} \\
1  & 00000b7c2cb2bc4b6cd9a96 & 2  &  00000dfc0c3f86787589a56 \\
3  & 00000dfc0c3f8a747985a96 & 4  &  00000f9e90729c9f4ca55a6 \\
5  & 00000fb24bcf48d26e55a56 & 6  &  00002671f06b3c7a41d7a96 \\
7  & 00003ba5d1f0b55ac1c7956 & 8  &  000044bb0787c2d92ed1596 \\
9  & 00004996a5f086ef03dc965 & 10 &  00004b67a135ca713cf2a56 \\
11 & 00004f2b6038d5ac19bc695 & 12 &  00004fe0fdc0a7a498b1695\\

\multicolumn{4}{c}{$n=26$} \\
1  & 000000ff0f846f1ca5a5aa955 & 2  & 00000b70c5f25257c69c39966  \\
3  & 00000b70cb54b0f1ea6239965 & 4  & 00000bab68f0da58e311d6a95  \\
5  & 00000c7e12e4391b865f8a596 & 6  & 00000f8f50cb26da9e51a9996  \\
7  & 00001477c0bed592960f39a55 & 8  & 000014bcf58a5f11269f05966  \\
9  & 0000178b0f2d9285badc19a66 & 10 & 000017ac6234e90b6d7d25966  \\
11 & 00002372d8f4a1ead7827b966 & 12 & 000027696c2491f88d3e0ba55\\

\multicolumn{4}{c}{$n=28$} \\
1  & 0000067cde3e50639ab46135aa5 & 2  & 000007f4038fa4d1529b16da656  \\
3  & 00000ab877e0a8fd862df0396a5 & 4  & 00000b344e59ca17f29216e5695  \\
5  & 00000df479ad14dab0c1f986a56 & 6  & 0000137872534b30ae5c2f69996  \\
7  & 000013847ef03e69586e2e96596 & 8  & 000015c86f122d54bb8fc4da5a5  \\
9  & 0000190ffe11a35f8695b709a96 & 10 & 00002799e66d6c8ebc25cf07aa5  \\
11 & 00003065e3788a2e1d693e4b556 & 12 & 00003a6b92877521ef412d1b956\\

\multicolumn{4}{c}{$n=30$} \\
1  & 000000f70b106f9d427a25e9a9695 & 2  & 000003f0ed871781d5d2a65876956  \\
3  & 000003f403872d2ba6cd5b1876a96 & 4  & 0000065f298b853ac3c2d86e39566  \\
5  & 000007e6883ca99f22570f0ae55a5 & 6  & 000007f701bd8f28b1a2583ae9a56  \\
7  & 00000bf4a07ab28c7dcd63e8da696 & 8  & 00000e3a785942359c33e0f669aa5  \\
9  & 00000f0f1c3a662b3dc6a59669aa5 & 10 & 00000fb507b6a1c5b03ec70e69aa6  \\
11 & 00000fe87624da3ac70bdeda59a66 & 12 & 00000ff118f947513c26d8a565a56 \\

\multicolumn{4}{c}{$n=32$} \\
1  & 00000138f64f1c1e77844f26d95a596 & 2  &0000067c7a5e84b6c1deb0cd71eaa65  \\
3  & 000006d074e9e0fb056835f289d55a6 & 4  &00000718f80fcfd24abb8925c9e6a95  \\
5  & 0000077403f8b0791e4ed89713e9565 & 6  &000007f30b587fc61bbe123969355a6  \\
7  & 0000093c5353ce49d36a4f50b516a96 & 8  &000009f4306ad6f086f92cb7d8c96a6  \\
9  & 00000b34d13a7d09c960d6790ada566 & 10 &00000d3dc8b2c4afaf078dd8678a596  \\
11 & 00000e6780dc4bb702f1b441fc96965 & 12 &00000eb38c5f53827c9e70716156995 \\

\end{tabular}
\\
\end{center}

\rmfamily

The complete list of Tuyrn-type sequences for $n \leq 32$ is available electronically at http://www.cs.uleth.ca/TurynType/.

In addition to the canonical $TT(n)$, the maximum number of initial 
zeroes in our canonical form is also of interest (Note that a zero in the canonical form represents a column of $1$s in the Turyn-type sequences). If a method for predicting the number of initial zeroes could be brought to light, it could greatly decrease future computation for individual Turyn-type sequences, since the first portion of each sequence would be pre-computed. Note that the first entry in each table listed above for $n \le 32$ represents the maximal number of initial zeroes for their respective lengths. There seems to be little correlation between $n$ and the maximum number of initial zeros.

By setting $x=1$ in (\ref{norm-TT}), we see that $6n-2$ is necessarily a sum of six (integer) squares as follows:
$$ A(1)^2+B(1)^2+2C(1)^2+2D(1)^2=6n-2. $$
It is noteworthy that our computation shows that for all even $n\le 32$,  any choice of 
four squares $A(1)^2$, $B(1)^2$, $C(1)^2$, $D(1)^2$ satisfying 
this equation can be realized by some $TT(n)$. 

For the sake of completeness, let us mention that $TT(n)$
for $n=26,28,30,32,34$ were constructed in \cite{KS},
and for $n=36$ in \cite{KT}. When transformed to the canonical form
(and encoded) these six sequences are:

\begin{eqnarray*}
&& 0560110f0f9ec89d54a6867dc \\
&& 0005189b4d2e583e5571efc9196 \\
&& 00788193c52741c99e060a73a22d5 \\
&& 005088b3dc4d69db0a13438a6c2e916 \\
&& 052351540cf016cfbe5809958b32825bc \\
&& 000f0f51c9bbd750cb048e3902185ca6a96
\end{eqnarray*}

\noindent The first four of them indeed occur in our complete listings
of class representatives for $n=26,28,30,32$.

\section{The Computational Method}\label{tt(38)}

As we have observed, there is compelling computational evidence that  $TT(n)$ exist for all even $n$. While our computational findings positively confirm the existence of $TT(n)$, they also show the difficulty in finding these sequences for large $n$.

In this section, we describe our method of finding a $TT(38)$ and 
set the stage for more computational work in trying to find $TT(n)$ 
for $n\ge 40$.

In order to search for $TT(38)$, we modified the search method in \cite{KT}. For the sake of completeness, we will briefly describe our  modified search method below.
\\ \\
{\bf The search method:}

We first find and retain
all partial sequences
\begin{center}
$$
\begin{array}{rcccl}
A^*&=&(a_1,\ldots,a_6,a_7,*,\ldots,*,a_{32},\ldots,a_{38}); \\
B^*&=&(b_1,\ldots,b_6,b_7,*,\ldots,*,b_{32},\ldots,b_{38}); \\
C^*&=&(c_1,\ldots,c_6,c_7,*,\ldots,*,c_{32},\ldots,c_{38});\\
D^*&=&(d_1,\ldots,d_6,*,*,\ldots,*,d_{32},\ldots,d_{37})
\end{array}
$$
\end{center}
for which
$$(N_{A^*}+N_{B^*}+2N_{C^*}+2N_{D^*})(s)=0 \ \ \textrm{for}\ \
s\ge 31,$$ 
and which have the canonical form detailed in Definition \ref{KanFor}.  In order to maintain a feasible number of cases, we precomputed $14,14,14,12$ entries in $A,B,C,D$ respectively. 
There are $23472940$ solutions in total. 
The set $\mathcal{S}$ of all of these solutions is input 
for the (modified) algorithm described in \cite[Section 3]{KT}.

To begin, select $a$, $b$, $c$, $d$ such that
$$a^2+b^2+2c^2+2d^2=226.$$

\noindent Generate all sequences $C$ with the sum of entries equal
to $c$ and for which $f_C(\theta)=N_C(0)+2\sum_{j=1}^{n-1}N_C(j)\cos j\theta \le 113$ for all $\theta \in
\{\frac{j\pi}{600}\ | \ j=1,2,\ldots,600\}$ and save
proper sequences according to their identical first and last seven
entries. We do the same for the sequences $D$ with the sum $d$.

The rest of the procedure is similar to the algorithm in \cite[Section 3]{KT}. Choose a solution $\{A^*,B^*,C^*,D^*\}$ in $\mathcal{S}$.
Let $\mathcal{C}$ and $\mathcal{D}$ be the sets of those sequences 
$C$ and $D$ whose first and last seven entries are identical
to the first and last seven entries of $C^*$ and $D^*$,
respectively. For any $C\in \mathcal{C}$ and $D\in \mathcal{D}$
for which $f_C(\theta)+f_D(\theta)\le 113$ for all $\theta \in
\{\frac{j\pi}{600}\ | \ j=1,2,\ldots,600\}$, we proceed to fill in
 partial sequences $A^*$ and $B^*$ step by step (see \cite[Section 3, Step 3]{KT} for further details) until we find appropriate sequences $A,B$. If we do not 
find such sequences, then we start again from the beginning.

Our search
with $a=8$, $b=-4$, $c=8$ and $d=-3$ resulted in the following
solution:

\noindent
$A=~$\Zp\Zp\Zp\Zp\Zm\Zm\Zp\Zp\Zp\Zp\Zp\Zm\Zp\Zp\Zp\Zm\Zm\Zm\Zp\Zm\Zp\Zp\Zm\Zp\Zp\Zp\Zp\Zp\Zm\Zp\Zp\Zm\Zm\Zm\Zm\Zm\Zm\Zp \\
$B=~$\Zp\Zm\Zp\Zp\Zp\Zm\Zm\Zm\Zm\Zp\Zp\Zm\Zp\Zm\Zp\Zp\Zm\Zm\Zm\Zm\Zm\Zm\Zm\Zm\Zp\Zm\Zm\Zm\Zp\Zp\Zp\Zm\Zp\Zm\Zp\Zp\Zm\Zp \\
$C=~$\Zp\Zp\Zp\Zm\Zp\Zm\Zp\Zp\Zp\Zp\Zp\Zm\Zp\Zp\Zp\Zm\Zp\Zm\Zm\Zm\Zm\Zp\Zp\Zp\Zm\Zp\Zm\Zm\Zp\Zm\Zm\Zp\Zp\Zp\Zm\Zp\Zp\Zm \\
$D=~$\Zp\Zm\Zm\Zp\Zp\Zm\Zm\Zm\Zp\Zp\Zm\Zm\Zp\Zp\Zm\Zp\Zm\Zm\Zm\Zm\Zp\Zm\Zp\Zm\Zm\Zm\Zp\Zm\Zp\Zp\Zp\Zp\Zm\Zp\Zm\Zm\Zp

In encoded form: $05128f55401f041adf7f65c53567822c9cb9c$.

The same algorithm was used to find all representatives of canonical forms. To give an indication
of computation time, all $TT(20)$ in canonical form took under five CPU minutes, whereas, all canonical $TT(32)$ took
approximately 50000 CPU hours on a single computer.

For the sake of completeness we recall the following known facts.
{\em Base sequences} $(A;B;C;D)$ are quadruples of $\{\pm 1\}$-sequences,
with $A$ and $B$ of length $m$ and $C$ and $D$ of length $n$,
and such that
\begin{equation} \label{norm}
N(A)+N(B)+N(C)+N(D)=2(m+n).
\end{equation}
See  \cite{DZ1, DZ2, DZ3} for details and the classification of these sequences.

Four $\{0,\pm 1\}$-sequences $A$, $B$, $C$, $D$ of length $n$ are
called $T$-{\it sequences} if
$$(N_A+N_B+N_C+N_D)(s)=0, \ \ \textrm{for}\ \  s\ge 1,$$
and in each position, exactly one of the entries of $A$, $B$, $C$,
$D$ is nonzero.

If $(A;B;C;D)$ are $TT(n)$, then $(C,D;C,-D;A;B)$ 
are base sequences of lengths $2n-1$,$2n-1$,$n$,$n$, respectively. 
(Here we use comma as the concatenation operator.)
Hence, the existence of $TT(38)$ implies the existence of base 
sequences of lengths $75$, $75$, $38$, $38$. 

If $(A;B;C;D)$ are base sequences of lengths $m$,$m$,$n$,$n$ 
respectively, then 
$$ ((A+B)/2,0_n;(A-B)/2,0_n;0_m,(C+D)/2;0_m,(C-D)/2) $$
are $T$-sequences of length $m+n$. 
(Here, the addition and subtraction of two sequences is 
component-wise, and $0_m$ denotes the sequence of $m$ zeroes.)
Consequently, $T$-sequences of length $75+38=113$ exist and 
we have the following corollary.

\begin{corollary}
There are base sequences of lengths $75$, $75$, $38$, $38$ and
therefore $T$-sequences of length $113$.
\end{corollary}

The existence of $T$-sequences of length $113$ implies the existence of an infinite number of Hadamard matrices; see \cite[Section 1]{KT} for details.

\section{Acknowledgments}

D.\v{Z}. \DJo{} and H. Kharaghani are grateful to NSERC for the continuing support of their research. Part of this work was made possible by the facilities
of the Shared Hierarchical Academic Research Computing Network
(SHARCNET:www.sharcnet.ca), WestGrid and Compute/Calcul Canada. The authors would like to thank an anonymous referee for a detailed list of suggestions which improved the composition and the presentation of the paper considerably.

\end{document}